\documentclass[11pt]  {amsart} 
\usepackage{amsmath,latexsym,amssymb,amsmath, 
amscd,amsthm,amsxtra}

\newtheorem{theorem}{Theorem}[section]
\newtheorem{prop}{Proposition}[section]
\newtheorem{lemma}{Lemma}[section]

\newtheorem{remark}{\textbf{Remark}}[section]
\def\r{\mathbb R}
\def\ss{\mathbb{S}}

\def\tr{\mathrm{tr}}

\def\p{\partial}

\def\a{\alpha}
\def\b{\beta}
\def\g{\gamma}
\def\k{\kappa}

\def\l{\lambda}
\def\L{\Lambda}
\def\s{\sigma}

\def\beq{\begin{eqnarray}}
\def\eeq{\end{eqnarray}}
\def\p{\partial}

\def\<{\langle}
\def\>{\rangle}

\def\n{\nabla}

\def\w{\mathcal{W}}

\begin{document}

\title[Inverse anisotropic mean curvature flow]{Inverse anisotropic mean curvature flow and a Minkowski type inequality}

\author{Chao Xia}\address{School of Mathematical Sciences, Xiamen University, 361005, Xiamen, China}
 \email{chaoxia@xmu.edu.cn}
\thanks{Research of CX is  supported by NSFC (Grant No. 11501480), the Fundamental Research Funds for the Central Universities (Grant No. 20720150012) and  the Natural Science Foundation of Fujian Province of China (Grant No. 2017J06003). }

\begin{abstract}In this paper, we show that the inverse anisotropic mean curvature flow in $\r^{n+1}$, initiating from a star-shaped, strictly $F$-mean convex hypersurface, 
exists for all time and after rescaling the flow converges exponentially  fast to a rescaled Wulff shape in the $C^\infty$ topology. As an application, we prove a Minkowski type inequality for star-shaped, $F$-mean convex hypersurfaces.
\end{abstract}

\maketitle

\section{Introduction}

Let $X(\cdot, t): M\times [0,T)\to \r^{n+1}$ be a family of smooth closed hypersurfaces in $\r^{n+1}$ satisfying 
\begin{eqnarray}\label{imcf}
\frac{\p}{\p t} X(x,t)=\frac{1}{H(x,t)}\nu(x,t),
\end{eqnarray}
where $H$ is the mean curvature function and $\nu$ is the outward unit normal. \eqref{imcf} is the so-called inverse mean curvature flow (IMCF). Gerhardt \cite{Ge1} and Urbas \cite{Ur} independently showed that, starting from a smooth closed, star-shaped and mean-convex  hypersurface, the flow \eqref{imcf} has a unique smooth solution for all time and the rescaled hypersurfaces $\tilde{X}(\cdot,t)=e^{-\frac1nt}X(\cdot,t)$ converges exponentially fast to a sphere.  Huisken-Ilmanen \cite{HI1, HI2} also defined a notion of weak solution for \eqref{imcf} and proved the higher regularity properties.

Besides the behavior of the flow \eqref{imcf} has been investigated in different ambient spaces \cite{Di, Ge2, Ge3, Sch}, the IMCF has been found to be a powerful tool to prove geometric inequalities. For example, Guan-Li \cite{GL} used the fully nonlinear version of the IMCF to prove the classical Alexandrov-Fenchel inequality for the quermassintegrals for star-shaped hypersurfaces. Huisken-Ilmanen \cite{HI1} used IMCF in the asymptotically flat manifolds to prove the Penrose inequality. More recently, Brendle-Hung-Wang \cite{BHW} used the IMCF in the Anti-de Sitter Schwarzschild manifolds to prove a Minkowski type inequality, which was applied to prove a Gibbons-Penrose's inequality in Schwarzschild spacetime \cite{BW}.

In this paper, we investigate the following inverse anisotropic mean curvature flow (IAMCF) in $\r^{n+1}$:
\begin{eqnarray}\label{iamcf}
\frac{\p}{\p t} X(\cdot,t)=\frac{1}{H_F(x,t)}\nu_F(x,t),
\end{eqnarray}
where $H_F$ is the anisotropic mean curvature function and $\nu_F$ is the outward anisotropic unit normal. Here we just mention that the anisotropy is determined by a given smooth closed strictly convex hypersurface $\w\subset \r^{n+1}$, which we call ``Wulff shape". $F\in C^\infty(\ss^n)$ indicating the  support function of $\w$ satisfies that the spherical Hessian is positive definite.  Geometrically, anisotropy is an alternative way of speaking about the relative geometry or the Minkowski geometry, which was intensively studied by Minkowski, Fenchel, etc., see e.g. \cite{BF} and the references therein. $\w$ was named as an ``Eichk\"orper" by Minkowski in the relative geometry. For the exact definition of $H_F$ and $\nu_F$ we refer to Sections \ref{sec2} and \ref{sec3}.
 
For an anisotropic flow,  the speed function depends not only on the usual curvature function of the evolved hypersurface but also its normal vector. For the anisotropic mean curvature flow, there are works concerning with weak solutions and their regularity issue, as well as its numerical analysis, see \cite{CGG, Gi} and the references therein.  
Simultaneously, much attention has been paid to the anisotropic curve flow  in $\r^{2}$ in the last decades after Angenent and Gurtin's modeling the motion of the interface with external force, see for example \cite{AG1, AG2, Gu} and the reference therein. For free external force, the flow has a natural interpretation as curve-shortening problem in Minkowski geometry and a complete picture has been captured by Gage \cite{Gage}, Gage-Li \cite{GaLi} and Chou-Zhu \cite{CZ1, CZ2}. General powers of anisotropic curve flows have been investigated by Andrews \cite{An2}. 

Comparatively, there is less work on higher dimensional anisotropic flows concerning about detailed convergence.  To the best of our knowledge, the only results in this direction are about the anisotropic Gauss curvature type flow  and the volume preserving anisotropic mean curvature flow considered by Andrews \cite{An3, An4}. In comparison to the isotropic flow,  it is harder to get the a priori estimate due to the anisotropy from the PDE point of view, and the behavior of geometric quantities in the anisotropic case is  worse from the geometric point of view.

\

Let us return to the IAMCF \eqref{iamcf}. The picture for the curve case is clear for strictly convex curves by the work of Andrews \cite{An2}. Among others, he proved that the flow \eqref{iamcf} in $\r^{2}$ exists for all time and converges to $\w$ at infinite time. 

The first aim of this paper is about the existence and convergence of higher dimensional IAMCF. We will show  the anisotropic version of Gerhardt and Urbas' result for star-shaped and $F$-mean convex hypersurface. A hypersurface $M\subset\r^{n+1}$ is called strictly F-mean convex if the anisotropic mean curvature $H_F>0$.
Our main result is the following

\begin{theorem}\label{thm}
Let $\w\subset \r^{n+1},$ $n\geq 2,$  be a given smooth closed strictly convex hypersurface containing the origin whose support function is $F:\ss^n\to\r$.  Let  $X_0: M^n\to \r^{n+1}$ be a smooth closed hypersurface which is star-shaped with respect to the origin and strictly $F$-mean convex. Then there exists a unique, smooth solution  $X(\cdot,t)$  to \eqref{iamcf} 
for $t\in [0,\infty)$ such that $X(\cdot,0)=X_0$. Moreover, the rescaled hypersufaces $e^{-\frac{t}{n}}X(\cdot,t)$ converge exponentially fast  to $\a_0 \w$ in the $C^\infty$ topology, where $\a_0=\int_M F(\nu(X_0))d\mu_{X_0}$ is the anisotropic area of $X_0$.\end{theorem}

The inverse anisotropic curvature flow has been considered by Ben Andrews in his dissertation \cite{An1}. There he showed up to $C^1$ estimate under certain conditions on the speed function, which excludes the IAMCF. 

Due to the anisotropy, most of the classical approach to prove the a priori estimates by Gerhardt and Urbas fails. Particularly, when we write the flow function as a scalar function of the graph function $\rho$ over $\ss^n$, the evolution equation for   $|\nabla^\ss \rho|^2$ does not behave well. Also, the evolution equation for the largest principal curvature is quite bad. 

To overcome these difficulties, we introduce a new Riemannian metric $\hat{g}$ on $M$, induced from a new Riemannian metric $G$ (See Section \ref{sec3}) on $\r^{n+1}$. This is inspired by a previous work of Andrews \cite{An4}. This is the key point of this paper. The new metric easies a lot  the $C^1$ estimate, but not for the $C^2$ estimate. We utilize the special structure of the anisotropic mean curvature and apply the classical theory from quasilinear elliptic and parabolic PDEs to our flow equation to get directly the $C^{2,\a}$ estimate.

To prove the convergence, we prove two quantities are monotone along the flow. By integration of these two quantities among all time, we find that  the limiting hypersurface must be anisotropically umbilical and has $F$ as its support function, which yields our convergence result.

\

The second aim of this paper is  to prove a  geometric inequality by using the IAMCF. This is also a motivation for us to consider the IAMCF.

The anisotropic curvature integrals have an direct relation with some special mixed volumes in the theory of convex bodies. An excellent book for the theory of convex bodies is by Schneider \cite{Sc}.
For any two convex bodies $K$ and $L$ in $\r^{n+1}$, the Minkowski sum is defined by $$(1-t)K+tL:=\{(1-t)x+ ty |x\in K, y\in L, t\in [0,1]\}.$$ Minkowski proved that the volume of $(1-t)K+tL$ is a polynomial in $t$, the coefficients of which are some mixed volumes. Precisely,
$$\hbox{Vol}\left((1-t)K+tL\right)
= \sum_{k=0}^{n+1}\binom{n+1}{k}(1-t)^{n+1-k}t^kV_{(k)}(K,L).$$
Especially, $V_{(0)}(K,L)=\hbox{Vol}(K)$ and $V_{(n+1)}(K,L)=\hbox{Vol}(L)$.

The most general Alexandrov-Fenchel inequality (see e.g.\cite{Sc}, Section 7.3, (7.54)) implies the following Minkowski type inequality (see e.g.\cite{Sc}, Section 7.3, (7.63)):
\begin{eqnarray*}
V_{(j)}(K,L)^{k-i}\geq V_{(i)}(K,L)^{k-j}V_{(k)}(K,L)^{j-i}, \quad \hbox{ for }0\leq i<j<k\leq n+1.
\end{eqnarray*}
In particular, for $k=n+1$, 
\begin{eqnarray}\label{AF0}
V_{(j)}(K,L)^{n+1-i}\geq V_{(i)}(K,L)^{n+1-j}\hbox{Vol}(L)^{j-i}, \quad \hbox{ for }0\leq i<j<n+1.
\end{eqnarray}

Assume that $\p L=\w$ is a smooth, strictly convex hypersurface and  $\p K$ is of $C^2$. 
We can interpret $V_{(i)}(K,L)$ in terms of the anisotropic curvature integrals (see e.g.\cite{BF}, 38 (13)):
\begin{eqnarray}\label{KL}
V_{(i)}(K,L)=\frac{1}{(n+1)\binom{n}{i-1}}\int_{\p K}\sigma_{i-1}(\k^F)  F(\nu)d\mu_{\p K}, \quad i=1,\cdots,n,
\end{eqnarray}
where $\sigma_i(\k^F)$ is the $i$-th elementary symmetric function on the anisotropic principal curvature $\k^F$.
When $L=B,$ the unit ball, 
$$V_{(i)}(K,B)=\frac{1}{(n+1)\binom{n}{i-1}}\int_{\p K}\sigma_{i-1}(\k)  d\mu_{\p K}, \quad i=1,\cdots,n,$$ where $\k$ is the usual principal curvature.
Therefore,  it makes sense to define $V_{(i)}(K,L)$ through \eqref{KL} for non-convex $K$ with $C^2$ boundary.


It is interesting to establish the Alexandrov-Fenchel and the Minkowski type inequalities for non-convex domains. Several works  in this direction have appeared, see for example \cite{Tr, GMTZ, GL, CW}.
In \cite{GL}, Guan-Li used Gerhardt and Urbas' result on the inverse curvature flow to show \eqref{AF0} holds true when $\w=\ss^n$ $(L=B)$ and $\p K$ is star-shaped and $k$-convex. 
In the same spirit of \cite{GL}, using the result on the IAMCF, Theorem \ref{thm}, we are able to show a special Minkowski type inequality,  \eqref{AF0} for $i=1$ and $j=2$, when $\p K$ is star-shaped and $F$-mean convex.

\begin{theorem}\label{AF}
Let $\w\subset\r^{n+1},$ $n\geq 2,$ be a smooth closed strictly convex hypersurface with support function $F$. Let $L$ be the enclosed domain by $\w$.
For any smooth star-shaped, $F$-mean convex ($H_F\geq 0$) hypersurface $M\subset \r^{n+1}$ which encloses $K$, we have
\begin{eqnarray}\label{AFI}
V_{(2)}(K,L)^n \geq V_{(1)}(K,L) ^{{n-1}} \hbox{Vol}(L),
\end{eqnarray}
for $V_{(i)}(K,L)$ defined by \eqref{KL}.
Equality in \eqref{AFI} holds if and only if $M$ is a rescaling and translation of $\w$.
\end{theorem}


\

The rest of the paper is organized as follows. In Section \ref{sec2}, we define the anisotropic mean curvature and give some variational formulae. In Section \ref{sec3}, we introduce Andrews' reformulation of the anisotropic curvature and give several fundamental properties. In Section \ref{sec4}, we study the IAMCF and prove the a priori estimates and the exponential convergence. In Section \ref{sec5}, we prove the Minkowski inequality \eqref{AFI} for star-shaped hypersurfaces. In Section \ref{sec6}, we give some discussion on other inverse anisotropic curvature flows.



\

\section{Anisotropic mean curvature}\label{sec2}

Given  a smooth closed strictly convex hypersurface  $\w\subset \r^{n+1}$ enclosing the origin, 
 the support function of $\w$, which is defined by  \begin{eqnarray*}
F(x)=\sup_{X\in \w}\<x,X\>_{g_{euc}}, \quad x\in \ss^n,
\end{eqnarray*}
 is a smooth positive function on $\ss^n$.  We recall several well known facts for a $C^2$ convex hypersurface and its support function, see e.g. \cite{Sc}, Section 2.5. $\w$ can be represented  by $F$ as 
 \begin{eqnarray}\label{ww}
\w=\{\phi(x)\in \r^{n+1}|\phi(x)=F(x)x+\nabla^{\ss} F(x),x\in \ss^n\},
\end{eqnarray}
  where $\nabla^{\ss}$ denotes the covariant derivative on $\ss^n$, see e.g. \cite{An'}, Eq. (2.10). Let $A_F: \ss^n\to \L^2T^* \ss^n$ be a 2-tensor defined by \begin{eqnarray*}
A_F(x)=\nabla^{\ss}\nabla^{\ss}F(x)+F(x)\s \hbox{ for }x\in\ss^n,
\end{eqnarray*}
 where $\s$ denotes the round metric on $\ss^n$. The  strictly convexity of $\w$ implies that $A_F$ is positive definite. It is well-known that the eigenvalues of $A_F$ with respect to $\s$ are the principal radii of $\w$. Note that $A_F$ is a Codazzi tensor on $\ss^n$.
 Conversely, given a smooth positive function $F$ on $\ss^n$ such that $A_F$ is positive definite, there is a unique smooth strictly convex hypersurface $\w$ given by \eqref{ww} whose support function is $F$.

Let $(M,g)$ be a smooth hypersurface in $\r^{n+1}$ with induced metric $g$ from $g_{euc}$, and $\nu: M\to \ss^n$ be its Gauss map. The anisotropic Gauss map of $M$  is defined by $$\begin{array}{lll}\nu_F: &&M\to  \w\\
&&X\mapsto \phi(\nu(X))=F(\nu(X))\nu(X)+\nabla^\ss F(\nu(X)).\end{array} $$
The anisotropic principal curvature $\k^F=(\k^F_1,\cdots, \k^F_n)$ of $M$ with respect to $\w$ at $X\in M$  is defined as the eigenvalues of $$d\nu_F: T_X M\to T_{\nu_F(X)} \w.$$
In particular, the anisotropic mean curvature of $M$ with respect to $\w$ at $X\in M$  is $$H_F(X):=\sum_{i=1}^n \k^F_i=\tr(d\nu_F)=\tr(A_F(\nu(X))\circ d\nu_X).$$
If we denote by $g_{ij}$ and $h_{ij}$ the first and the second fundamental form of $M\subset \r^{n+1}$ respectively, then in local coordinates, $$H_F(X)=\sum_{i,j,k=1}^n {A_F}_i^j(\nu(X))g^{ik}(X)h_{kj}(X).$$
Here we view $A_F$ as a $(1,1)$-tensor on $\ss^n$.

An important variational characterization for $H_F$ is that it arises from the first variation of the parametric area functional $\int_M F(\nu)d\mu_g$.
Similarly, we have a variational formula for the total anisotropic mean curvature functional.
\begin{prop}[Reilly \cite{Re1, Re2}]\label{HL} Let $X_0: M\to\r^{n+1}$ be a smooth closed, oriented hypersurface and $X(\cdot,t)$ be a variation of $X_0$ with variational vector field $\frac{\p}{\p t}X(\cdot,t)=\psi(X)\nu(X)$, where $\psi\in C^{\infty}(M)$. Then
\begin{eqnarray}\label{var0}
\frac{d}{d t}\int_M F(\nu)d\mu_g=\int_M H_F(X) \psi(X) d\mu_g,
\end{eqnarray}
\begin{eqnarray}\label{var}
\frac{d}{d t}\int_M H_F(X) F(\nu)d\mu_g=\int_M 2\sigma_{2}(\k^F(X)) \psi(X) d\mu_g,
\end{eqnarray}
where $$\sigma_2(\k^F)=\sum_{i<j} \k^F_i\k^F_j.$$ 
\end{prop}

The variational formulae \eqref{var0} and \eqref{var} may be familiar to experts. When $F=1$,  such formulas are well-known, see e.g. Reilly \cite{Re1}. For general $F$, Reilly \cite{Re2} derived the variational formula for $\int_M \sigma_k F(\nu)d\mu$ for any $k$, see also He-Li \cite{HL}. Here we give a proof for the case $H_F$ for the convenience of readers.

\begin{proof}By the tensorial property, we do not distinguish upper and lower indexes in the proof whenever applicable.
Since $\p_t \nu=-\nabla \psi$ and $\p_td\mu_g=H\psi d\mu_g$,  we have by integration by parts that
\begin{eqnarray*}
\frac{d}{d t}\int_M F(\nu)d\mu_g&=&\int_M -\nabla_p^\ss F(\nu)\nabla_p \psi +F(\nu) H \psi
\\&=&\int_M \nabla_q^\ss\nabla_p^\ss F(\nu)h_{pq} \psi +F(\nu) H\psi \\&=&\int_M H_F \psi .
\end{eqnarray*}
Here $H$ is the usual mean curvature of $M\subset \r^{n+1}$.

We also have $\p_t {h_i^j}=-\nabla_i\nabla_j \psi-\psi h_{ik} h_{kj}$. Therefore
\begin{eqnarray}\label{var1}
&&\frac{d}{d t}\int_M H_F(X) F(\nu)d\mu\\&=&\int_M-\nabla^\ss_p A_{ij}(\nu)\nabla_p \psi h_{ij} F(\nu)+A_{ij}(\nu)(-\nabla_i\nabla_j \psi-\psi h_{ik} h_{kj})F(\nu)\psi\nonumber\\&&+\int_M-H_F \nabla^\ss_p F(\nu)\nabla_p \psi+H_FF(\nu) H\psi .\nonumber
\end{eqnarray}
 Since $A$ is Codazzi tensor on $\ss^n$, by integration by parts,
 \begin{eqnarray}\label{var2}
&&\int_M -A_{ij}(\nu)\nabla_i\nabla_j \psi F(\nu)\\&=&\int_M \nabla^\ss_p A_{ij}(\nu)h_{ip} \nabla_j \psi F(\nu)+A_{ij}(\nu)\nabla_j \psi \nabla_p^\ss F(\nu)h_{ip} \nonumber\\&=&\int_M \nabla^\ss_p A_{ij}(\nu)h_{ij} \nabla_p \psi F(\nu)+ A_{ij}(\nu)\nabla_j \psi \nabla_p^\ss F(\nu)h_{ip} .\nonumber
\end{eqnarray}
Integrating by parts again, we have
\begin{eqnarray}\label{var3}
&&\int_M A_{ij}(\nu)\nabla_j \psi \nabla_p^\ss F(\nu)h_i^p=\int_M -\left(\nabla_j(A_{ij}(\nu)h_{ip})\nabla_p^\ss F(\nu)+ A_{ij}(\nu)h_{ip}\nabla_p^\ss \nabla_q^\ss F(\nu) h_{jq}\right)\psi, 
\end{eqnarray}
\begin{eqnarray}\label{var4}
&&\int_M -H_F \nabla^\ss_p F(\nu)\nabla_p \psi= \int_M \left(\nabla_p H_F \nabla^\ss_p F(\nu)+ H_F\nabla^\ss_p\nabla_q^\ss F(\nu)h_{pq}\right) \psi.
\end{eqnarray}
Combining \eqref{var1}--\eqref{var4}, we deduce
\begin{eqnarray}\label{var5}
&&\frac{d}{d t}\int_M H_F(X) F(\nu)d\mu\\&=&\int_M \left(\nabla_p H_F-\nabla_j(A_{ij}(\nu)h_{ip})\right) \nabla^\ss_p F(\nu)\psi\nonumber\\&&+\int_M \left(-A_{ij}(\nu)h_{ip}\nabla_p^\ss \nabla_q^\ss F(\nu) h_j^q-A_{ij}(\nu) h_{ik} h_{kj} F(\nu)\right)\psi \nonumber\\&&+\int_M \left(H_F\nabla^\ss_p\nabla_q^\ss F(\nu)h_{pq}+H_FF(\nu) H\right) \psi \nonumber\\
&=&I+II+III.\nonumber\end{eqnarray}
We easily see that
\begin{eqnarray}\label{var6}
II+III&=&\int_M -A_{ij}(\nu)A_{pq}(\nu)h_{ip}h_{jq}\psi +H_FA_{pq}(\nu)h_{pq}\psi \\&=&  \int_M (H_F^2-|\k^F|^2)\psi=\int_M 2\sigma_2(\k^F) \psi .\nonumber
\end{eqnarray}
On the other hand, since $A$ is Codazzi on $\ss^n$ and $h$ is Codazzi on $X$, we have
\begin{eqnarray*}
\nabla_j(A_{ij}(\nu)h_{ip})&=&\nabla^\ss_q A_{ij}(\nu)h_{ip}h_{jq}+A_{ij}(\nu)\nabla_j h_{ip}\\&=&\nabla^\ss_i A_{jq}(\nu)h_{ip}h_{jq}+A_{ij}(\nu)\nabla_p h_{ij}\\&=&\nabla_p(A_{ij}(\nu)h_{ij})=\nabla_p H_F.
\end{eqnarray*}
Thus $I=0$. The assertion follows from \eqref{var5} and \eqref{var6}.
\end{proof}


\

\section{Andrews' formulation of anisotropic curvatures}\label{sec3}

In this section we recall Andrews' formulation of anisotropic curvatures \cite{An4}. In \cite{X2}, we reformulated Andrews' idea in a more direct way. Here we  follow the notation in \cite{X2}.

As in Section \ref{sec2}, let $\w\subset \r^{n+1}$ be a smooth closed strictly convex hypersurface enclosing the origin, whose support function is $F\in C^\infty(\ss^n)$. We extend $F\in C^\infty(\ss^n)$ homogeneously to be a $1$-homogeneous function $F\in C^\infty(\r^{n+1}\setminus\{0\})$ by $$F(x)=|x|F\left(\frac{x}{|x|}\right), \quad x\in \r^{n+1}\setminus \{0\}\hbox{ and } F(0)=0.$$
One can check easily that $F\in C^\infty(\r^{n+1}\setminus\{0\})$ is in fact a Minkowski norm in $\r^{n+1}$ in the sense that 
\begin{itemize}
\item[(i)] $F$ is a norm in $\mathbb{R}^{n+1}$, i.e.,  $F$ is a convex, $1$-homogeneous function satisfying $F(x)>0$ when $x\neq 0$;
\item[(ii)] $F$ satisfies a uniformly elliptic condition: $D^2 (\frac12 F^2)$ is positive definite in $\mathbb{R}^{n+1}\setminus \{0\}$.
\end{itemize} 
Here $D$ is the Euclidean gradient and $D^2$ is the Euclidean Hessian.
In fact, (ii) is  equivalent that  $(\n^\ss\n^\ss F+F\s)$ is positive definite on $(\ss^n,\s)$.  (see e.g. \cite{X1}, Proposition 1.4).

For a Minkowski norm $F\in C^\infty(\r^{n+1}\setminus\{0\})$, the dual norm of $F$ is defined as  $$F^0(\xi):=\sup_{x\neq 0}\frac{\langle x,\xi\rangle}{F(x)},\quad \xi\in \mathbb{R}^{n+1}.$$ $F^0$ is also a Minkowski norm lying in $C^\infty(\r^{n+1}\setminus\{0\})$, see e.g. \cite{Shen}, Lemma 3.1.2.

We introduce a Riemannian metric $G$ with respect to $F^0$ in $T\r^{n+1}$:
\begin{eqnarray*}
&&G(\xi)(V,W):=\sum_{\a,\b=1}^{n+1}\frac{\p^2 \frac12(F^0)^2(\xi)}{\p \xi^\a\p \xi^\b} V^\a W^\b, \quad\hbox{ for } \xi\in \mathbb{R}^{n+1}\setminus \{0\}, V,W\in T_\xi{\r^{n+1}}.
\end{eqnarray*}

Since $F^0$ is in general not quadratic, the third derivative of $F^0$ does not vanish. We set
\begin{eqnarray*}
&&Q(\xi)(U,V,W):=\sum_{\alpha,\beta,\gamma=1}^{n+1} Q_{\alpha\beta\gamma}(\xi)U^\alpha V^\beta W^\gamma:=\sum_{\alpha,\beta,\gamma=1}^{n+1} \frac{\partial^3(\frac12(F^0)^2(\xi)}{\partial \xi^\alpha \partial \xi^\beta \partial \xi^\gamma}U^\alpha V^\beta W^\gamma,
\end{eqnarray*}
for $\xi\in \mathbb{R}^{n+1}\setminus \{0\},$ $U,V,W\in T_\xi{\r^{n+1}}.$

When we restrict the metric $G$ to $\mathcal{W}$,  the $1$-homogeneity of $F^0$ tells us
\begin{eqnarray*}
&G(\xi)(\xi,\xi)=1,  G(\xi)(\xi, V)=0, \quad \hbox{ for } \xi\in \w, V\in T_\xi \w.
\\&Q(\xi)(\xi, V, W)=0, \quad \hbox{ for } \xi\in\w, V, W\in \r^{n+1}.
\end{eqnarray*}
We remark that for deducing above formulae, we need also to use the fact $\mathcal{W}=\{\xi\in \r^{n+1}: F^0(\xi)=1\}$.

Let us now return to a hypersurface $M\subset \r^{n+1}$. 
The anisotropic normal is defined by $\nu_F=F(\nu)\nu+\nabla^\ss F.$ It follows from the $1$-homogeneity of $F$ that
\begin{eqnarray*}
\nu_F= DF(\nu).
\end{eqnarray*} 
Since $\nu_F(X)\in \w$ for $X\in M$, we have 
\begin{eqnarray*}
&G(\nu_F)(\nu_F,\nu_F)=1,  G(\nu_F)(\nu_F, V)=0, \quad \hbox{ for } V\in T_X M, 
\\&Q(\nu_F)(\nu_F, V, W)=0, \quad \hbox{ for } V, W\in \r^{n+1}.
\end{eqnarray*}
This means  $\nu_F(X)$ is perpendicular to $T_X M$ with respect to the metric $G(\nu_F)$.
This motivates us to define \begin{eqnarray*}
\hat{g}(X):=G(\nu_F(X))|_{T_X M}, \quad X\in M
\end{eqnarray*}
 as a Riemannian metric on $M\subset \mathbb{R}^{n+1}$. We denote by $\hat{D}$ and $\hat{\nabla}$ the Levi-Civita connections of $G$ on $\r^{n+1}$ and $\hat{g}$ on $M$ respectively.

As in Section 2, the anisotropic principal curvature $\k^F$ of $M\subset \mathbb{R}^{n+1}$ with respect to $\w$ is defined as the eigenvalues of  $$d\nu_F: T_X M\to T_{\nu_F(X)} \w.$$

Using $G$ and $\hat{g}$, we can reformulate $\k^F$ and $H_F$ as follows.
Denote by $\hat{g}_{ij}$ and $\hat{h}_{ij}$ the first and the second fundamental form of $(M, \hat{g})\subset (\r^{n+1}, G)$, i.e.,
$$\hat{g}_{ij}=G(\nu_F(X))(\p_i X, \p_j X),\quad \hat{h}_{ij}=G(\nu_F(X))(\hat{D}_{\p_i }\nu_F, \p_j X),$$
Then $\k^F$ is the  eigenvalues of $(\hat{g}^{ik}\hat{h}_{kj})$ and 
\begin{eqnarray*}
H_F=\sum_{i,j=1}^n \hat{g}^{ij}\hat{h}_{ij}.
\end{eqnarray*}

It is direct to see that for $M=\w$, we have $\nu_F(\w)=X(\w)$, $\hat{h}_{ij}=\hat{g}_{ij}$ and $H_F=n$.

For the previous reformulation, we have the following  anisotropic Gauss-Weingarten type formulae and the anisotropic Gauss-Codazzi type equation.
\begin{lemma}[Xia \cite{X2}, Lemma 2.5] \label{lem2-1}
 \begin{eqnarray}\label{Gauss}
\partial_i\partial_j X=-\hat{h}_{ij}\nu_F+\hat{\nabla}_{\partial_i } \partial_j +\hat{g}^{kl}A_{ijl}  \partial_k X; \hbox{ (Gauss formula)}\end{eqnarray}
\begin{eqnarray}\label{Weingarten}
\partial_i \nu_F=\hat{g}^{jk}\hat{h}_{ij}\partial_k X; \hbox{ (Weingarten formula) }
\end{eqnarray}
 \begin{eqnarray}\label{Gausseq}
\hat{R}_{ijkl}&=&\hat{h}_{ik}\hat{h}_{jl}-\hat{h}_{il}\hat{h}_{jk}+\hat{\nabla}_{\partial_l} A_{jki}-\hat{\nabla}_{\partial_k } A_{jli}\\&&+\hat{g}^{pm}A_{jkp} A_{mli}-\hat{g}^{pm}A_{jlp} A_{mki}; \hbox{ (Gauss equation) }\nonumber
\end{eqnarray}
\begin{eqnarray}\label{Codazzi}
\hat{\nabla}_k\hat{h}_{ij}+\hat{h}_j^l A_{lki}=\hat{\nabla}_j\hat{h}_{ik}+\hat{h}_k^l A_{lji}. \hbox{ (Codazzi equation) }
\end{eqnarray}
Here $\hat{R}$ is the Riemannian curvature tensor of $\hat{g}$, $A$ is a $3$-tensor
\begin{eqnarray}\label{AA}
A_{ijk}=-\frac12\left(\hat{h}_i^l Q_{jkl}+\hat{h}_j^l Q_{ilk}-\hat{h}_k^l Q_{ijl}\right),
\end{eqnarray} where $Q_{ijk}=Q(\nu_F)(\partial_i X, \partial_j X, \partial_k X)$.
\end{lemma}

\begin{remark}We remark here and in the subsequent calculations that, we regard $X$ and $\nu_F$ as vector-valued functions in $\r^{n+1}$ with a fixed Cartesian coordinate. Hence the terms $\partial_i\partial_j X$ and $\partial_i \nu_F$ are understood as the usual partial derivative on vector-valued functions.
\end{remark}
Note that the $3$-tensor $A$ on $(M,\hat{g})\to (\r^{n+1},G)$ depends on $\hat{h}_i^j$. It is direct to see that $Q$ is totally symmetric in all three indices, while $A$ is only symmetric for the first two indices.

Let us compare the previous formulae with the isotropic case. The Weingarten formula is in the same behavior, while the Gauss formula involves an extra tangential part $\hat{g}^{kl}A_{ijl}  \partial_k X$ besides the Levi-Civita connection part. The anisotropic Codazzi type equation implies $\hat{h}_{ij}$ is not a Codazzi tensor in $(M,\hat{g})$. The anisotropic Gauss type equation also includes messier terms involving $A$ and $\hat{\nabla} A$. These quite complicated formulae make the analysis of the anisotropic curvature problems much harder.

Let us write the anisotropic area element by $$d\mu_F:=F(\nu)d\mu_g.$$
In \cite{X2}, we proved an important property about the Laplacian operator $\hat{\Delta}$ with respect to $\hat{g}$ and $d\mu_F$, which will play an important role in this paper.
\begin{lemma}[Xia \cite{X2}, Lemma 2.8]\label{volume}
Let $d\mu_{\hat{g}}$ be the induced volume form of $(M,\hat{g})$. Assume that $$d\mu_F(X)=F(\nu(X))d\mu_g(X)=\varphi(X) d\mu_{\hat{g}}(X).$$ Then
\begin{equation*}
\hat{\nabla}_i \log\varphi= \hat{g}^{jk}A_{ijk}.
\end{equation*}
Consquently, for any $f\in C^\infty(M)$,
\begin{equation*}
\int_M \hat{\Delta} f+\hat{g}^{jk}A_{ijk}\hat{\nabla}^i f d\mu_F=0.
\end{equation*}
\end{lemma}

\

\section{Inverse anisotropic mean curvature flow}\label{sec4}

In this section we study the IAMCF \eqref{iamcf} initiating from a star-shaped, strictly F-mean convex hypersurface. 


Let us fix some notation.  We use $\nabla^\ss$  to denote the covariant derivative on the round sphere $(\ss^n,\s)$.  We use $g_{ij}$, $h_{ij}$, $\nabla$  to denote the first and the second fundamental form, the covariant derivative of $(M,g)\subset (\r^{n+1}, g_{euc})$,  while $\hat{g}_{ij}$, $\hat{h}_{ij}$ $\hat{\nabla}$ to denote that of $(M,\hat{g})\subset (\r^{n+1}, G)$, respectively. 

\

 It follows from $\nu_F=F(\nu)\nu+\nabla^\ss F(\nu)$ that up to a diffeomorphism of $M$, the flow \eqref{iamcf} is equivalent to 
 $$\p_t X=\frac{F(\nu)}{H_F}\nu.$$
 
Since $X_0$ is star-shaped with respect to the origin, we can write $X_0$ as a graph of a function  over $\ss^n$: $$X_0=\{(\rho_0(x),x): x\in \ss^n\}.$$

If each $X(\cdot, t)$ is star-shaped,  the evolved hypersurfaces can be reparametrized as  graphs  over $(\ss^n,\sigma)$:
$$X(x,t)=\rho(x,t)x, \quad x\in \ss^n,$$ 
where $\rho(x,t)$ is the graph function. Denote by $\g=\log \rho$. 
Then it is standard to derive
\begin{eqnarray*}
&&\nu=\frac{x-\nabla^\ss \g}{\sqrt{1+|\nabla^\ss \g|^2}},
\end{eqnarray*}
\begin{eqnarray*}
&& H_F=A_{ij}(\nu)\frac{1}{\rho \sqrt{1+|\nabla^\ss \g|^2}}\left[\delta_{ij}-\left(\sigma^{ik}-\frac{\g^i\g^k}{{1+|\nabla^\ss \g|^2}}\right)\g_{jk}\right],
\end{eqnarray*}
and the scalar parabolic equation for $\g$:
\begin{eqnarray}\label{g}
\frac{\p \g}{\p t}=\frac{\sqrt{1+|\nabla^\ss \g|^2}F}{\rho H_F}=\frac{(1+|\nabla^\ss \g|^2)F}{A_{ij}(\nu)\left[\delta_{ij}-\left(\sigma^{ik}-\frac{\g^i\g^k}{{1+|\nabla^\ss \g|^2}}\right)\g_{jk}\right]}.\end{eqnarray}
Note that here $$F=F\left(\frac{x-\nabla^\ss \g}{\sqrt{1+|\nabla^\ss \g|^2}}\right),\quad A_{ij}(\nu)=A_{ij}\left(\frac{x-\nabla^\ss \g}{\sqrt{1+|\nabla^\ss \g|^2}}\right).$$
The dependence of RHS of \eqref{g} on $\nu$ makes the a priori estimates subtle.

Equation \eqref{g} is a fully nonlinear parabolic equation. The short time existence is standard by using implicit function theorem. Without loss of generality, we assume that the flow exists for $[0,T)$ and $X(\cdot, t), t\in [0,T)$ is star-shaped. To prove the long time existence, we need to establish the a priori estimates independent of $T$ for \eqref{g}. However, it is quite complicated to work directly on \eqref{g} because of its dependence of $x$ as just mentioned. We mostly work on the original flow equation \eqref{iamcf}.

Before getting into the a priori estimates, Let us first derive some evolution equations for the flow \eqref{iamcf}. 
Let $$u:=\<X,\nu\>_{g_{euc}}$$ be the support function of $X(\cdot, t)$ and $$\hat{u}:=G(\nu_F)(\nu_F,X)$$ be the anisotropic support function of $X(\cdot, t)$. 
It is easy to see that
\begin{eqnarray}\label{relation u}
\hat{u}=\frac{u}{F(\nu)}.
\end{eqnarray}
Indeed, 
\begin{eqnarray*}
\hat{u}=G(\nu_F)(\nu_F,X)=\<DF^0(DF(\nu)),X\>_{g_{euc}}=\<\frac{\nu}{F(\nu)},X\>_{g_{euc}}=\frac{u}{F(\nu)},
\end{eqnarray*}
where we used $DF^0(DF(x))=\frac{x}{F(x)},$ see e.g. \cite{X1}, Proposition 1.3.
Equation \eqref{relation u} implies that there exists two constants $\l, \L$ depending only on $F$, such that 
\begin{eqnarray}\label{relation u2}
\l u\leq \hat{u}\leq \L u.
\end{eqnarray}


\begin{prop}\label{evolv}
Let $f=\frac{1}{H_F}$ be the speed function. Along the flow (\ref{iamcf}), we have the following evolution equations:

\begin{enumerate}
\item[(i)] $\nu_F$ evolves under
\begin{eqnarray}\label{evolve nu}
\p_t \nu_F=-\hat{\nabla} f;
\end{eqnarray}
\item[(ii)]  The anisotropic area form $d\mu_F$ evolves under \begin{eqnarray}\label{evolve mu}
\p_t d\mu_F=H_F fd\mu_F=d\mu_F,
\end{eqnarray}
\item[(iii)] $\hat{g}_{ij}$  evolves under \begin{eqnarray}\label{evolve g}
\p_t \hat{g}_{ij}=2f \hat{h}_{ij}-Q_{ijk}\hat{\nabla}^k f;
\end{eqnarray}
\item[(iv)]   $\hat{h}_{i}^j$ evolves under\begin{eqnarray}\label{evolve h}
\p_t \hat{h}_{i}^j=- f\hat{h}_i^k\hat{h}_{k}^j-\hat{\nabla}^j \hat{\nabla}_i f- \hat{g}^{jk}A_{pik}\hat{\nabla}^p f;\end{eqnarray}
\item[(v)] $H_F$ evolves under \begin{eqnarray}\label{evolve f}
\p_t H_F- \frac{1}{H_F^2}\left(\hat{\Delta} H_F+\hat{g}^{ik}A_{pik}\hat{\nabla}^p H_F\right)=-2\frac{|\hat{\nabla} H_F|^2_{\hat{g}}}{H_F^3}- \frac{1}{H_F}|\hat{h}|_{\hat{g}}^2;\end{eqnarray}
\item[(vi)]  $u$ evolves under
\begin{eqnarray}\label{evolve u}
\p_t \hat{u}-\frac{1}{H_F^2}\left(\hat{\Delta} \hat{u}+\hat{g}^{ik}A_{pik}\hat{\nabla}^p \hat{u}\right)= \frac{1}{H_F^2}|\hat{h}|_{\hat{g}}^2\hat{u}.
\end{eqnarray}
\end{enumerate}
\end{prop}

\begin{proof}
In the proof we will frequently use the property that
\begin{eqnarray}\label{3Q}
Q(\nu_F)(\nu_F, V, W)=0, \quad  \quad V, W\in \r^{n+1}.
\end{eqnarray}
 
(i) Taking derivative of $G(\nu_F)(\nu_F,\nu_F)=1$ and  $G(\nu_F)(\nu_F, X_i)=0$ with respect to $t$ and using the Weingarten formula \eqref{Weingarten} and \eqref{3Q}, we have
$$0=\p_tG(\nu_F)(\nu_F,\nu_F)=2G(\nu_F)(\p_t\nu_F,\nu_F)+Q(\nu_F)(\p_t \nu_F, \nu_F,\nu_F);$$
\begin{eqnarray*}
0&=&\p_tG(\nu_F)(\nu_F,X_i)\\&=&G(\nu_F)(\p_t\nu_F,X_i)+G(\nu_F)(\nu_F,\p_i(\p_t X))+Q(\nu_F)(\p_t \nu_F, \nu_F,X_i)\\
&=&G(\nu_F)(\p_t\nu_F,X_i)+G(\nu_F)(\nu_F,\p_i f\nu_F+f\p_i \nu_F)\\
&=&G(\nu_F)(\p_t\nu_F,X_i)+\p_i f.
\end{eqnarray*}
Thus $$\p_t \nu_F= -\hat{\nabla} f.$$

(ii) Let $\Omega$ be the Lebesgue volume form in $\r^{n+1}$. Then the area element $d\mu_g$ of $(M,g)$  can be interpreted in the local coordinates as
$$d\mu_g=\Omega(\nu,\p_1 X, \cdots, \p_n X) dx^1\cdots dx^n.$$
Hence $$d\mu_F=F(\nu)d\mu_g=\Omega(\nu_F,\p_1 X, \cdots, \p_n X) dx^1\cdots dx^n.$$
It follows from \eqref{evolve nu} and \eqref{Weingarten} that
\begin{eqnarray*}
\p_t d\mu_F&=&\left[\Omega(\p_t\nu_F,\p_1 X, \cdots, \p_n X) +\sum_{i=1}^n\Omega(\nu_F,\p_1 X, \cdots, \p_i(\p_t X),\cdots,\p_n X) \right]dx^1\cdots dx^n\\&=&\sum_{i=1}^n\Omega(\nu_F,\p_1 X, \cdots, \p_i(f\nu_F),\cdots,\p_n X) dx^1\cdots dx^n\\&=&f \hat{h}_i^i \Omega(\nu_F,\p_1 X, \cdots, \p_n X) dx^1\cdots dx^n\\&=&H_F fd\mu_F=d\mu_F.
\end{eqnarray*}

(iii)-(iv):Using the Gauss-Weingarten formula \eqref{Gauss}, \eqref{Weingarten} and \eqref{evolve nu}, 
we directly compute
\begin{eqnarray*}
\p_t \hat{g}_{ij}&=& \p_t G(\nu_F)(X_i,X_j)\\&=&G(\nu_F)(\p_i f\nu_F+f\p_i \nu_F,X_j)+G(\nu_F)(\p_j f\nu_F+f\p_j\nu_F,X_i)+Q(\nu_F)(\p_t \nu_F, X_i,X_j)\\
&=&fG(\nu_F)(\hat{h}_i^kX_k,X_j)+fG(\nu_F)(\hat{h}_j^kX_k,X_i)+Q(\nu_F)(-\hat{\nabla}^pfX_p, X_i,X_j)\\
&=& 2f \hat{h}_{ij}-Q_{ijp}\hat{\nabla}^pf;
\end{eqnarray*}
\begin{eqnarray*}
\p_t \hat{h}_{ij}&=& \p_t G(\nu_F)(\p_i X,\p_j \nu_F)
\\&=&G(\nu_F)(\hat{\nabla}_i (f\nu_F), \p_j \nu_F)+G(\nu_F)(\p_i X,-\hat{\nabla}_j (\hat{\nabla}^p fX_p))+Q(\nu_F)(\p_iX, \p_j \nu_F, -\hat{\nabla}^p f X_p)\\
&=&f\hat{h}_i^k\hat{h}_{jk}-\hat{\nabla}_j \hat{\nabla}_i f- A_{jpi}\hat{\nabla}^p f -Q_{ilp}\hat{h}_{j}^l \hat{\nabla}^p f 
\end{eqnarray*}

Thus 
\begin{eqnarray*}
\p_t \hat{h}_{i}^j&=& \p_t \hat{g}^{jk}\hat{h}_{ik}+\hat{g}^{jk}\p_t \hat{h}_{ik}
\\&=&- f\hat{h}_i^k\hat{h}_{k}^j-\hat{\nabla}^j \hat{\nabla}_i f+ \hat{g}^{jr}\hat{g}^{ks} \hat{h}_{ik}Q_{rsp}\hat{\nabla}^p f- \hat{g}^{jk}( A_{kpi}+Q_{ilp}\hat{h}_{k}^l )\hat{\nabla}^p f\\
&=&- f\hat{h}_i^k\hat{h}_{k}^j-\hat{\nabla}^j \hat{\nabla}_i f- \hat{g}^{jk}A_{pik}\hat{\nabla}^p f.
\end{eqnarray*}
In the last inequality we used \eqref{AA} to do the computation 
\begin{eqnarray*}
&&\hat{g}^{jr}\hat{g}^{ks}\hat{h}_{ik}Q_{rsp} - \hat{g}^{jk}( A_{kpi}+Q_{ilp}\hat{h}_{k}^l )
\\&=&\hat{g}^{jr} \hat{h}_{i}^sQ_{rsp} +\frac12 \hat{g}^{jk}(\hat{h}_k^q Q_{qpi}+\hat{h}_p^q Q_{kqi}-\hat{h}_i^qQ_{kpq})- \hat{g}^{jk}Q_{ilp}\hat{h}_{k}^l 
\\&=&\frac12\hat{g}^{jr} \hat{h}_{i}^sQ_{rsp} -\frac12 \hat{g}^{jk}\hat{h}_k^q Q_{qpi}+\frac12\hat{g}^{jk}\hat{h}_p^q Q_{kqi}\\&=&  -\hat{g}^{jk}A_{pik}.
\end{eqnarray*}

(v) Equation \eqref{evolve f} follows by taking trace of \eqref{evolve h}. 

(vi) Using \eqref{evolve nu} and \eqref{3Q}, we have
\begin{eqnarray*}
\p_t \hat{u}&=&\p_t G(\nu_F)(\nu_F,X)\\&=&G(\nu_F)(-\hat{\nabla} f,X)+G(\nu_F)(\nu_F, f\nu_F)+Q(\nu_F)(\p_t\nu_F,\nu_F,X)\\&=&- \hat{\nabla}^k fG(\nu_F)(X,X_k)+f.\end{eqnarray*}
Using the Weingarten formula \eqref{Weingarten} and \eqref{3Q}, we have
\begin{eqnarray*}
\hat{\nabla}_i\hat{u}&=& \hat{\nabla}_i G(\nu_F)(\nu_F,X)\\&=& G(\nu_F)(\hat{\nabla}_i\nu_F,X)+G(\nu_F)(\nu_F,X_i)+Q(\nu_F)(\hat{\nabla}_i\nu_F,\nu_F,X)
\\&=& \hat{h}_i^pG(\nu_F)( X_p,X).
\end{eqnarray*}
Using also the anisotropic Codazzi formula \eqref{Codazzi}, we have
\begin{eqnarray*}
&&\hat{\Delta} \hat{u} + \hat{g}^{ik}A_{pik}\hat{\nabla}^p\hat{u}\\&=&\hat{\nabla}^i [\hat{h}_i^pG(\nu_F)( X_p,X)]+ \hat{g}^{ik}A_{pik} \hat{h}^{pm}G(\nu_F)( X_m,X)\\&=&  \hat{\nabla}^i \hat{h}_i^pG(\nu_F)(X_p,X)+\hat{g}^{ik}A_{pik} \hat{h}^{pm}G(\nu_F)( X_m,X)\\&&+ \hat{h}_i^p\left[G(\nu_F) (-\hat{h}_p^i\nu_F+\hat{g}^{iq}\hat{g}^{rm}A_{pqr}X_m,X)+\delta_p^i+ Q(\nu_F)(\hat{h}^{iq} X_q,X_p,X)\right]\\&=&\left[\hat{\nabla}^p \hat{h}_i^i+\hat{g}^{ps}\hat{h}^{ir}A_{sri}-\hat{g}^{ir}\hat{h}^{ps}A_{rsi}\right]G(\nu_F)(X_p,X)-|\hat{h}|_{\hat{g}}^2\hat{u}+H_F\\&&+ \left[\hat{g}^{iq}\hat{g}^{mr}\hat{h}_i^pA_{pqr}+\hat{g}^{mr}\hat{h}_i^p\hat{h}^{iq}Q_{pqr}+\hat{g}^{ik}A_{pik} \hat{h}^{pm}\right]G(\nu_F)(X_m,X)\\
&=&\hat{\nabla}^p H_FG(\nu_F)(X_p,X)-|\hat{h}|_{\hat{g}}^2\hat{u}+H_F\\&&+ \left[A_{rpq}+A_{pqr}+\hat{h}_{p}^sQ_{sqr}\right] \hat{g}^{mr}\hat{h}^{pq}G(\nu_F)(X_m,X).\end{eqnarray*}
A direct computation using \eqref{AA} shows that 
$$A_{rpq}+A_{pqr}+\hat{h}_{p}^sQ_{sqr}=0.$$
Therefore, 
\begin{eqnarray*}
&&\hat{\Delta} \hat{u} + \hat{g}^{ik}A_{pik}\hat{\nabla}^p\hat{u}=\hat{\nabla}^p H_FG(\nu_F)(X_p,X)-|\hat{h}|_{\hat{g}}^2\hat{u}+H_F.
\end{eqnarray*}
It follows that
\begin{eqnarray*}
&&\p_t \hat{u}-\frac{1}{H_F^2}\left(\hat{\Delta} \hat{u} + \hat{g}^{ik}A_{pik}\hat{\nabla}^p\hat{u}\right)=\frac{1}{H_F^2}|\hat{h}|_{\hat{g}}^2\hat{u}.
\end{eqnarray*}

\end{proof}

\begin{remark}
We can reprove Proposition \ref{HL} in an alternative way by using Proposition \ref{evolv}.
Indeed, formula \eqref{var0} follows directly from \eqref{evolve mu}. Using \eqref{evolve h} and Lemma \ref{volume}, we see easily
\begin{eqnarray*}
\frac{d}{dt}\int_M H_Fd\mu_F&=&\int_M (- f|\hat{h}|_{\hat{g}}^2-\hat{\Delta} f- \hat{g}^{ik}A_{pik}\hat{\nabla}^p f)+H_F^2 fd\mu_F\\&=&\int_M 2\sigma_2(\k^F)fd\mu_F.
\end{eqnarray*}
\end{remark}

\

\

We are now in a position to prove the a priori estimates for the flow \eqref{iamcf}. Let $$\tilde{X}(\cdot,t)=e^{-\frac1n t}X(\cdot,t),$$  the rescaled hypersurfaces. In the following we use $\;\tilde{}\;$ to indicate the related geometric quantity of $\tilde{X}$.

The a priori bound for the graph function $\rho(\cdot, t)$ follows by comparing with the homothetic solutions.
\begin{prop}\label{C0}
There exist two positive constants $r$ and $R$, depending only on $X_0$, such that 
\begin{eqnarray*}
re^{\frac1n t}\leq |X(\cdot,t)|\leq Re^{\frac1n t} \hbox{ or }r\leq |\tilde{X}(\cdot,t)|\leq R.
\end{eqnarray*}

\end{prop}
\begin{proof}
Since $X_0$ is star-shaped and closed, we can find $r$ and $R$ such that $$r\w\subset X_0\subset R\w.$$ Since the anisotropic mean curvature of the hypersurface $\w$ is the constant $n$, and $\nu_F(\w)$ is the same as its position vector, we know the flow starting from $\w$ is homothetical. Hence one can solve explicitly the solution of  the flow starting from $r\w$ ($RW$ resp. ) as $r(t)\w$ ($R(t)\w$ resp.), where $r(t)=re^{\frac1n t}$ and $R(t)=Re^{\frac1n t}$. Since the flow is parabolic, by the comparison principle, we have $X(\cdot, t)$ is bounded by $r(t)\w$ from below and by $R(t)\w$ from above.
\end{proof}

We then prove the $C^1$ estimate. 
\begin{prop}\label{C1}
There exists some constant  $C$, depending on $F, r, R$ and $\|\nabla^\ss \g(\cdot, 0)\|$, such that 
\begin{eqnarray*}
|\nabla^\ss \g|(x,t)\leq C.
\end{eqnarray*}
\end{prop}

\begin{proof}
As we mentioned before, the evolution equation for $|\nabla^\ss \g|^2$ does not behave  well. 
We will use the evolution of $\hat{u}$.
In fact, we utilize $$\tilde{\hat{u}}=e^{-\frac1nt}\hat{u},$$ the anisotropic support function of the rescaled hypersurface $\tilde{X}=e^{-\frac1nt}X$. 

It follows from \eqref{evolve u} that
\begin{eqnarray}\label{scale u}
\p_t \tilde{\hat{u}}-\frac{1}{H_F^2}\left(\hat{\Delta} \tilde{\hat{u}}+\hat{g}^{ik}A_{pik}\hat{\nabla}^p \tilde{\hat{u}}\right)= \left(\frac{1}{H_F^2}|\hat{h}|_{\hat{g}}^2-\frac1n\right)\tilde{\hat{u}}.
\end{eqnarray}
The elementary Cauchy-Schwarz inequality tells that
\begin{eqnarray*}
\frac{1}{H_F^2}|\hat{h}|_{\hat{g}}^2-\frac1n\geq 0.
\end{eqnarray*}
Using the maximum principle on \eqref{scale u}, we see
\begin{eqnarray*}
\tilde{\hat{u}}(\cdot, t)\geq \min \tilde{\hat{u}}(\cdot,0)=\min \hat{u}(\cdot,0).
\end{eqnarray*}
which implies 
\begin{eqnarray*}
\hat{u}(\cdot, t)\geq e^{\frac1n t}\min \hat{u}(\cdot,0).
\end{eqnarray*}
In view of \eqref{relation u2}, we know that
 \begin{eqnarray*}
u(\cdot, t)\geq \frac{\l}{\L}e^{\frac1n t}\min u(\cdot,0).
\end{eqnarray*}
Since $u=\frac{\rho}{\sqrt{1+|\nabla^\ss \g|^2}}$, combining with the $C^0$ estimate we have
 \begin{eqnarray*}
|\nabla^\ss \g|(\cdot, t)\leq C,
\end{eqnarray*}
where $C$ depends on $F, r, R$ and $\|\nabla^\ss \g(\cdot, 0)\|.$ 
\end{proof}

Next we show the anisotropic mean curvature is uniformly bounded for $\tilde{X}(\cdot, t)$.

\begin{prop}\label{HF}
There exists some constant  $C$, depending on $F$   and the initial data of $X_0$, such that 
\begin{eqnarray}
\frac1C\leq \tilde{H}_F\leq C.
\end{eqnarray}
\end{prop}
\begin{proof}
From \eqref{evolve f} and \eqref{evolve u}, we have the following evolution equation 
\begin{eqnarray}\label{fu0}
\p_t \left(H_F\hat{u}\right)- \frac{1}{H_F^2}\left(\hat{\Delta} \left(H_F\hat{u}\right)+\hat{g}^{ik}A_{pik}\hat{\nabla}^p \left(H_F\hat{u}\right)\right)+\frac{2}{H_F^3} \hat{\nabla}^i H_F\hat{\nabla}_i\left(H_F\hat{u}\right)=0.\end{eqnarray}
It follows from the maximum principle that
\begin{eqnarray*}
\min H_F\hat{u}(\cdot,0)\leq H_F\hat{u}(\cdot,t)\leq \max H_F\hat{u}(\cdot,0).\end{eqnarray*}
$H_F\hat{u}$ is scaling invariant, so 
\begin{eqnarray*}
\min \tilde{H}_F\tilde{\hat{u}}(\cdot,0)\leq \tilde{H}_F\tilde{\hat{u}}(\cdot,t)\leq \max \tilde{H}_F\tilde{\hat{u}}(\cdot,0).\end{eqnarray*}
The assertion now follows from Proposition \ref{C1}.
\end{proof}

In view of Proposition \ref{C0}--\ref{HF}, we see that $\tilde{\rho}$, $\tilde{u}$ and $\tilde{H}_F$ are uniformly bounded from above and below by positive constants. Therefore, we see readily that equation \eqref{g} is uniformly parabolic.
However, because the equation \eqref{g} is fully nonlinear, we still need the $C^2$ estimate of $\tilde{\rho}$. 

It is quite hard to use the evolution equation for $\hat{h}_{i}^j$, because the anisotropy brings technical difficulties. Here we realize that the anisotropic mean curvature is itself a quasilinear operator and we utilize several estimates from the theory of quasilinear elliptic or parabolic equations.
In the following we denote by $C^{k,\alpha}$ the spatial H\"older space and $\tilde{C}^{k,\alpha}$ the space-time H\"older space, see e.g. \cite{Ur} page 361.

If we write $\tilde{\gamma}=\log \tilde{\rho}$, then $\tilde{H}_F$ can be expressed in terms of $\tilde{\g}$: 
\begin{eqnarray*}
\tilde{H}_F=\frac{1}{\tilde{\rho}\sqrt{1+|\nabla^\ss \tilde{\g}|^2}} A_{ij}(\tilde{\nu})\left[\delta_{ij}-\left(\s^{ik}-\frac{\tilde{\g}^i\tilde{\g}^k}{1+|\nabla^\ss \tilde{\g}|^2}\right)\tilde{\g}_{jk}\right].
\end{eqnarray*}
Hence
\begin{eqnarray}\label{ellip}
A_{ij}(\tilde{\nu})\left(\s^{ik}-\frac{\tilde{\g}^i\tilde{\g}^k}{1+|\nabla^\ss \tilde{\g}|^2}\right)\tilde{\g}_{jk}
=\sum_i A_{ii}(\tilde{\nu})- \tilde{\rho}\sqrt{1+|\nabla^\ss \tilde{\g}|^2}\tilde{H}_F.
\end{eqnarray}
Since $|\nabla^\ss \tilde{\g}|$ and $\tilde{H}_F$ are uniformly bounded,  \eqref{ellip} is a uniformly elliptic equation.

Note that $$\tilde{\nu}=\frac{x-\nabla^\ss \tilde{\g}}{\sqrt{1+|\nabla^\ss \tilde{\g}|^2}}.$$

We write \eqref{ellip} as a general form of quasilinear equations:
\begin{eqnarray}\label{ellip1}
a_{ij}(x,\nabla^\ss \tilde{\g})\tilde{\g}_{ij}+b(x,\tilde{\g},\nabla^\ss\tilde{\g})=0.\end{eqnarray}
We note that $a_{ij}\in  C^1(M\times \mathbb{R}^n), b\in  C^0(M\times \mathbb{R}\times\mathbb{R}^n)$ and we have the structural condition for \eqref{ellip1}:
\begin{eqnarray*}
a_{ij}(x,\nabla^\ss \tilde{\g})\xi^i\xi^j\geq \l |\xi|^2>0,\quad \forall \xi\in \r^{n+1}\setminus\{0\},
\end{eqnarray*}
\begin{eqnarray*}
|a_{ij}(x,p)|+|D_{x_k} a_{ij}(x,p)|+|D_{p_k} a_{ij}(x,p)|+|b(x,z,p)|\leq \L,
\end{eqnarray*}
where $\l$ and $\L$ depend only on $\|\tilde{\g}\|_{C^1}$.
It follows from \cite{GT}, Chapter 13, Theorem 13.6 that $\|\nabla^\ss \tilde{\g}\|_{C^\a}\leq C.$  In turn, $\tilde{\hat{u}}$ has a $C^\alpha$ bound in $x$.

\

Next we show that $\|\tilde{H}_F\|_{C^\b}\leq C$ for some $\b\in (0,1).$
In order to prove this, we look at the equation for $$P:=H_F \hat{u}.$$ We recall from \eqref{fu0} that $P$ satisfies
\begin{eqnarray}\label{22}
\p_t P- \frac{1}{H_F^2}\left(\hat{\Delta} P+\hat{g}^{ik}A_{pik}\hat{\nabla}^p P\right)+\frac{2}{H_F^3} \hat{\nabla}^i H_F\hat{\nabla}_i P=0.\end{eqnarray}
The key observation is that equation \eqref{22} is a quasilinear parabolic equation of divergence form  on the weighted manifold $(M, \hat{g}, d\mu_F=\varphi d\mu_{\hat{g}}).$
We will use the classical parabolic PDE theory (\cite{LSU}) to prove the H\"older continuity of $P$.

Let $\zeta\in C_c^\infty(B_\rho\times [0,T))$ be some cut-off function with values in $[0,1]$ in some small ball $B_\rho\subset M$. Multiplying equation \eqref{22} with $\zeta^2P$, integrating by parts over $X(\cdot, t)\times [t_0,t]$ for any $[t_0,t]\subset [0,T)$ and using Lemma \ref{volume}, we obtain
 \begin{eqnarray}\label{33}
&&\int_{t_0}^{t} \frac{d}{dt} \int_{B_\rho} \frac12\zeta^2 P^2 d\tilde{\mu}_F dt -\int_{t_0}^{t}  \int_{B_\rho} \zeta\p_t\zeta P^2 d\tilde{\mu}_F dt\\&=&\int_{t_0}^{t} \int_{B_\rho} \p_t P\cdot\zeta^2Pd\tilde{\mu}_F dt\nonumber\\&=& \int_{t_0}^{t}\int_{B_\rho} \frac{\zeta^2P}{H_F^2}\left(\hat{\Delta} P+\hat{g}^{ik}A_{pk i}\hat{\nabla}^p P\right)-\frac{2\zeta^2P}{H_F^3} \hat{\nabla}^i H_F\hat{\nabla}_i P d\tilde{\mu}_F dt\nonumber
\\&=& \int_{t_0}^{t}\int_{B_\rho} -\frac{\zeta^2}{{H}_F^2}|\hat{\nabla} P|_{\hat{g}}^2 + \frac{2\zeta P}{{H}_F^2}\hat{\nabla} ^i\zeta\hat{\nabla} _i Pd\tilde{\mu}_F dt. \nonumber
\end{eqnarray}
In the first equality we also used 
\begin{eqnarray*}
\p_t d\tilde{\mu}_F =\p_t (e^{-t}d{\mu}_F)= e^{-t}(\p_t d\mu_F-d\mu_F)=0.
\end{eqnarray*}

By using the H\"older inequality in \eqref{33}, we have
 \begin{eqnarray}\label{44}
&&\int_{B_\rho} \frac12P^2\zeta^2d\tilde{\mu}_F \bigg|_{t_0}^{t} +\int_{t_0}^{t}\int_{B_\rho} \frac{\zeta^2}{2H_F^2}|\hat{\nabla} P|_{\hat{g}}^2 d\tilde{\mu}_F dt
\\&\leq &\int_{t_0}^{t} \int_{B_\rho}  \frac{1}{{H}_F^2}|\hat{\nabla} \zeta|_{\hat{g}}^2P^2+ |\zeta||\p_t\zeta| P^2 d\tilde{\mu}_F dt.\nonumber
\end{eqnarray}

Note that
\begin{eqnarray*}
\tilde{\hat{g}}_{ij}&=&G(\tilde{\nu}_F)(\tilde{X}_i,\tilde{X}_j)=\frac{\p^2 \frac12(F^0)^2}{\p \xi^\a\p \xi^\b}(\nu_F)\tilde{X}_i^\a\tilde{X}_j^\b.\end{eqnarray*}
Because $F^0$ is a Minkowski norm, there exists a constant $C$, depending only on $F$, such that
\begin{eqnarray*}
\frac1C\<\tilde{X}_i, \tilde{X}_j\> \leq\tilde{\hat{g}}_{ij}\leq C\<\tilde{X}_i, \tilde{X}_j\>.\end{eqnarray*}
On the other hand, due to the $C^1$ estimate,
\begin{eqnarray*}
\frac1C\sigma_{ij}\leq \<\tilde{X}_i, \tilde{X}_j\>\leq C\sigma_{ij}.\end{eqnarray*}
Hence $\tilde{\hat{g}}_{ij}$ and $d\tilde{\mu}_F$ is uniformly bounded.
Also from Proposition \ref{HF}, $\tilde{H}_F$ is uniformly bounded.  We find that estimate \eqref{44} is in a similar behavior as \cite{LSU}, Chapter V, (1.13). From the argument after (1.13) there,  locally  our quantity $P$ belongs to the space $\mathcal{B}_2$ in \cite{LSU}, Chapter II. Therefore, by \cite{LSU}, Chapter II, Theorem 8.1, we obtain that 
$$\|P\|_{\tilde{C}^{\g}}\leq C,$$ for some $\g\in (0,1)$. 
Particularly, since $$P=H_F \hat{u}=\tilde{H}_F \tilde{\hat{u}}$$ and $\tilde{\hat{u}}$ has a $C^\a$ bound in $x$, we conclude that $\tilde{H}_F$ has a $C^{\b}$ bound in $x$ for some $\b\in (0,1)$. 

We return to equation \eqref{ellip} and find that both the coefficient and the RHS have some H\"older continuous bound. It follows from the classical elliptic Schauder theory that 
$$|\tilde{\g}|_{C^{2,\alpha}(\ss^n\times [0,T))}\leq C\hbox{ for some }\a\in (0,1).$$

From  \eqref{g} we know $\p_t \tilde{\g}$ is uniformly bounded. Therefore
$$|\tilde{\g}|_{\tilde{C}^2(\ss^n\times [0,T))}\leq C.$$

Now we have an uniformly  parabolic and concave equation \eqref{g} for  scalar function $\tilde{\g}$ with the a priori $\tilde{C}^{2}$ bound (in space-time).  By standard fully non-linear parabolic PDE theory, we will have all the higher order a priori estimates and consequently the long time existence of the solution. Moreover, all the geometric quantities and their derivatives for $\tilde{X}$ are uniformly bounded.

\

It is left to show the convergence of the flow \eqref{iamcf}.

Let $\tilde{\k}^F(x,t)$ be the anisotropic principal curvatures of $\tilde{X}(x,t)$. We know from our a priori estimates that $\tilde{\k}^F(x,t)$ is uniformly bounded for all $t\in [0,+\infty)$.

Denote $$\mathcal{H}(t):=\int_M \tilde{H}_Fd\tilde{\mu}_F,\quad t\in [0,\infty).$$
We deduce from \eqref{var} that along the flow \eqref{iamcf},
\begin{eqnarray}\label{vv}
\frac{d}{dt}\mathcal{H}(t)&=&\frac{d}{dt}\left\{e^{\frac{1-n}{n}t}\int_M H_Fd\mu_F\right\}\\&=&e^{\frac{1-n}{n}t}\left(\int_M \frac{1-n}{n} H_F+2\sigma_2(\k^F) \frac{1}{H_F}d\mu_F\right)\nonumber\\&=& \int_M \left(\frac{2\sigma_2(\tilde{\k}^F)}{\tilde{H}_F}-\frac{n-1}{n}\tilde{H}_F\right)d\tilde{\mu}_F\nonumber\\&=& -\int_M \frac{1}{\tilde{H}_F} \left| \tilde{\hat{h}}_i^{j}-\frac{\tilde{H}_F}{n}\delta_{ij}\right|^2 d\tilde{\mu}_F\leq 0.\nonumber
\end{eqnarray}

Integrating \eqref{vv} over $[0,T]$,
\begin{eqnarray}\label{aa}
\mathcal{H}(0)-\mathcal{H}(T)=\int_0^T \int_M  \frac{1}{\tilde{H}_F} \left|\tilde{\hat{h}}_i^{j}-\frac{\tilde{H}_F}{n}\delta_{ij}\right|^2 d\tilde{\mu}_Fdt.
\end{eqnarray}
Since $\mathcal{H}(T)> 0$ for all $T<\infty$, we see from \eqref{aa} that
\begin{eqnarray}\label{a1}
0\leq \int_0^{\infty} \int_M   \frac{1}{\tilde{H}_F} \left|\tilde{\hat{h}}_i^{j}-\frac{\tilde{H}_F}{n}\delta_{ij}\right|^2 d\tilde{\mu}_Fdt\leq \mathcal{H}(0)\leq C.
\end{eqnarray}
The integrand in \eqref{a1} is uniformly continuous in $t$. Hence
\begin{eqnarray*}
\int_M  \left|\tilde{\hat{h}}_{i}^j-\frac{\tilde{H}_F}{n}\delta_{ij}\right|^2 d\tilde{\mu}_F\to 0\hbox{ as }t\to \infty .
\end{eqnarray*}

It follows from the regularity estimates and the interpolation theorem that 
\begin{eqnarray}\label{umbilic}
 \left|\tilde{\hat{h}}_{i}^j-\frac{\tilde{H}_F}{n}\delta_{ij}\right|^2 \to 0\hbox{ uniformly in } C^\infty\hbox{ as }t\to \infty.
\end{eqnarray}

On the other hand, from the anisotropic Codazzi formula \eqref{Codazzi}, we have
\begin{eqnarray*}
\hat{\nabla}_j \tilde{\hat{h}}_{i}^j=\hat{\nabla}_i \tilde{H}_F+\tilde{\hat{h}}^{jl}\tilde{A}_{lij}-\tilde{\hat{h}}_i^l \tilde{A}_{ljj}.
\end{eqnarray*}
Thus 
\begin{eqnarray}\label{x1}
|\hat{\nabla}_j\tilde{\hat{h}}_{i}^j-\hat{\nabla}_i \tilde{H}_F|_{\tilde{\hat{g}}}\leq C\sum_j |\tilde{\k}^F_i-\tilde{\k}^F_j|, \hbox{ for any }i.
\end{eqnarray}
We see from \eqref{umbilic} that 
\begin{eqnarray}\label{x2}
 |\hat{\nabla}_j \tilde{\hat{h}}_{i}^j-\frac1n\hat{\nabla}_i \tilde{H}_F|_{\tilde{\hat{g}}}\to 0\hbox{ uniformly as }t\to \infty \hbox{ for any }i, 
\end{eqnarray}
and 
\begin{eqnarray}\label{x3}
|\tilde{\k}^F_i-\tilde{\k}^F_j|\to 0\hbox{ uniformly as }t\to \infty \hbox{ for any }i\neq j.  
\end{eqnarray}
From \eqref{x1}-\eqref{x3} we deduce that 
\begin{eqnarray*}
|\hat{\nabla}\tilde{H}_F|_{\tilde{\hat{g}}}\to 0\hbox{ uniformly as }t\to \infty.
\end{eqnarray*}
It follows that
\begin{eqnarray}\label{Hconv}
\tilde{H}_F-n\k_0\to 0\hbox{ uniformly in } C^\infty\hbox{ as }t\to \infty.
\end{eqnarray}
with some positive constant $\k_0$.

We will show next $P:=H_F \hat{u}$ converges to a constant. Note that $P$ is scaling invariant. Denote by $$\mathcal{P}(t):=\int_M P d\tilde{\mu}_F.$$
Let us recall the evolution equation \eqref{fu0} for $P$:
\begin{eqnarray}\label{PP}
\p_t P- \frac{1}{H_F^2}\left(\hat{\Delta} P+\hat{g}^{ik}A_{pik}\hat{\nabla}^p P\right)+\frac{2}{H_F^3} \hat{\nabla}^i H_F\hat{\nabla}_i P=0.
\end{eqnarray}
Integrating by parts with respect to $d\tilde{\mu}_F$, we have
\begin{eqnarray*}
\frac{d}{dt} \mathcal{P}(t)=0.
\end{eqnarray*}
This means $\mathcal{P}(t)=\mathcal{P}^*$ is a constant.
On the other hand, multiplying  \eqref{PP} by $P$ and integrating by parts, we obtain
\begin{eqnarray*}
&&\p_t \int_M \frac12 |P(\cdot,t)-\mathcal{P}^*|^2  d\tilde{\mu}_F\\&=&\p_t \int_M \frac12 P(\cdot,t)^2-\frac12(\mathcal{P}^*)^2  d\tilde{\mu}_F\\&=& -\int_M \frac{1}{\tilde{H}_F^2}|\hat{\nabla} P|_{\tilde{\hat{g}}}^2 d\tilde{\mu}_F
\\&\leq &- C\int_M \frac12 |P(\cdot,t)-\mathcal{P}^*|^2  d\tilde{\mu}_F. \end{eqnarray*}
In the last inequality we used the boundedness of $\tilde{H}_F$ and the Poincar\'e inequality.


It follows that
\begin{eqnarray*}
&&\int_M \frac12 |P(\cdot,t)-\mathcal{P}^*|^2  d\tilde{\mu}_F\leq  Ce^{-Ct}. \end{eqnarray*}
The standard argument using the interpolation theorem, see e.g. \cite{Ur} page 371, yields that
\begin{eqnarray}\label{Pconv}
&&\|P(\cdot,t)-\mathcal{P}^*\|_{C^k}\leq  Ce^{-Ct}\hbox{ for any }k\ge 0. \end{eqnarray}

Combining \eqref{Hconv} and \eqref{Pconv},
we see that
\begin{eqnarray}\label{uconv}
\left\|\tilde{\hat{u}}(\cdot,t)- \frac{\mathcal{P}^*}{n\k_0}\right\|_{C^k}\to 0\hbox{ for any }k\ge 0.
\end{eqnarray}
Note that we do not have exponential convergence for $\tilde{H}_F$. We can not get exponential convergence of $\tilde{\hat{u}}$ from \eqref{Pconv}. From \eqref{Hconv} and \eqref{uconv}, it is clear that $\mathcal{P}^*=n$.

To show the exponential convergence, we shall write the flow equation as a scalar equation for the anisotropic support function on $\w$ for $t$ large. Because
for $t$ large enough, the evolved hypersurfaces are strictly convex, we can reparametrize  $X(\cdot, t): \w^n\to \mathbb{R}^{n+1}$ by its inverse anisotropic Gauss map $\nu_F^{-1}$. The anisotropic principal curvatures of $X$ are equal to the eigenvalues of the inverse of $$U_{ij}:=\hat{\nabla}^\w_i \hat{\nabla}^\w_j\hat{u}-\frac12 Q_{ijk} \hat{\nabla}^\w_k\hat{u}+\hat{u}\delta_{ij},$$ where $\hat{\nabla}^\w$ is the covariant derivative with respect to $\hat{g}$ on $\w$. See \cite{X2}. The anisotropic support function $\hat{u}$, viewed as functions on $\w$, satisfies
\begin{eqnarray*}
\p_t \hat{u}=\frac{1}{H_F}=\frac{\sigma_n}{\sigma_{n-1}}(U_{ij}),
\end{eqnarray*}
 \begin{eqnarray}\label{cc}
\p_t \tilde{\hat{u}}=\frac{1}{\tilde{H}_F}-\frac{\tilde{\hat{u}}}{n}=\frac{\sigma_n}{\sigma_{n-1}}(\tilde{U}_{ij})-\frac{\tilde{\hat{u}}}{n}.
\end{eqnarray}
Let $$\mathcal{U}(t):=\int_\w  \tilde{\hat{u}}(\cdot,t) d\mu_F,$$

\eqref{cc} and \eqref{Pconv} tells us \begin{eqnarray}\label{ccc}
\quad \left|\frac{d}{dt}\mathcal{U}\right|\leq  Ce^{-Ct}.
\end{eqnarray}

It follows from \eqref{ccc} that  there exists a constant $\mathcal{U}^*$ such that 
\begin{eqnarray}\label{ddd}
\|\mathcal{U}(t)-\mathcal{U}^*\|\leq  Ce^{-Ct}.
\end{eqnarray}

On the other hand, using \eqref{cc}, \eqref{ccc}, Lemma \ref{volume} and the Poincar\'e inequality, we deduce
\begin{eqnarray*}
&&\frac{d}{dt}\int_\w |\tilde{\hat{u}}(\cdot,t)-\mathcal{U}(t)|^2 d\mu_F
\\&=&\int_\w 2\tilde{\hat{u}}\left[\frac{\sigma_n}{\sigma_{n-1}}(\tilde{U}_{ij})-\frac{\tilde{\hat{u}}}{n}\right]-2\mathcal{U}(t)\frac{d}{dt}\mathcal{U} d\mu_F
\\&\leq &\int_\w \frac2n\tilde{\hat{u}}\left(\hat{\Delta}^\w\tilde{\hat{u}}-\frac12 Q_{iik} \hat{\nabla}^\w_k\tilde{\hat{u}}\right)d\mu_F+Ce^{-Ct}
\\&= &-\int_\w \frac2n|\hat{\nabla}^\w \tilde{\hat{u}}|_{\hat{g}_\w}^2 d\mu_F+Ce^{-Ct}
\\&\leq &-C\int_\w |\tilde{\hat{u}}(\cdot,t)-\mathcal{U}(t)|^2 d\mu_F+Ce^{-Ct}.
\end{eqnarray*}
Thus \begin{eqnarray}\label{eee}
\int_\w |\tilde{\hat{u}}(\cdot,t)-\mathcal{U}(t)|^2 d\mu_F\leq  Ce^{-Ct}.
\end{eqnarray}

Combining \eqref{ddd} and \eqref{eee}, and using the interpolation theorem, we see that 
\begin{eqnarray*}
\|\tilde{\hat{u}}(\cdot,t)-\mathcal{U}^*\|_{C^k(\w)} \leq  Ce^{-Ct} \hbox{ for any } k\ge 0.
\end{eqnarray*}

Therefore,  we proved that  $\tilde{\hat{u}}:\w\to\r$, as the anisotropic support function of $\tilde{X}$, converges  exponentially  to a constant in the $C^\infty$ topology.  Note from \eqref{relation u} that  $$\tilde{\hat{u}}(y, t)=\frac{\tilde{{u}}(x,t)}{F(x)}, \quad\hbox{ for } x\in \ss^n,\quad y=DF(x)\in \w. $$ Thus $\tilde{{u}}:\ss^n\to\r$, as the usual support function of $\tilde{X}$, converges exponentially  to $F:\ss^n\to\r$ in the $C^\infty$ topology.
Since a strictly convex hypersurface is uniquely determined by its support function as \eqref{ww}, we conclude that
$\tilde{X}$ converges exponentially fast  to a rescaling of $\w$ in the $C^\infty$ topology, without any correction by translations.
The proof of Theorem \ref{thm} is complete.

\

\section{A Minkowski type inequality}\label{sec5}

In this section we prove Theorem \ref{AF}.

Assume first $M$ is strictly $F$-mean convex,
Let $X(\cdot, t), t\in [0,\infty)$ be the solution of \eqref{iamcf} with $X(\cdot,0)=M$ and $\tilde{X}(\cdot,t)=e^{-\frac1nt}X(\cdot,t)$. Theorem \ref{thm} tells that $\tilde{X}(\cdot,t)$ converges smoothly to a rescaling of $\w$, say $\a_0 \w$.
We see from \eqref{aa} and \eqref{vv} that
\begin{eqnarray}\label{aaa}
\frac{d}{dt}\int_{\tilde{X}} d\tilde{\mu}_F=0.
\end{eqnarray}
\begin{eqnarray}\label{bbb}
\frac{d}{dt}\int_{\tilde{X}} \tilde{H}_Fd\tilde{\mu}_F\leq 0.
\end{eqnarray}
Therefore, using \eqref{aaa} and \eqref{bbb}, \begin{eqnarray*}
&&\int_X \tilde{H}_Fd\tilde{\mu}_F\geq \int_{\a_0 \w} H_F(\a_0 \w)d\mu_F
=\frac{n}{\a_0}\int_{\a_0 \w} d\mu_F
\\&=&n\left(\int_\w d\mu_F\right)^\frac1n\left(\int_{\a_0\w} d\mu_F\right)^{\frac{n-1}{n}}\\&=&n\left(\int_\w d\mu_F\right)^\frac1n\left(\int_{\tilde{X}} d\tilde{\mu}_F\right)^{\frac{n-1}{n}}.
\end{eqnarray*}
On the other hand, 
\begin{eqnarray*}
\int_\w d\mu_F=\int_\w F(\nu) d\mu_g=\int_{\ss^n} F(x)\det (A_F)d\mu_{\ss^n}=(n+1)\hbox{Vol}(L).
\end{eqnarray*}
Therefore, at $t=0$, we have
\begin{eqnarray*}
&&\frac1n\int_M H_FF(\nu)d\mu_g\geq \left((n+1)\hbox{Vol}(L)\right)^\frac1n\left(\int_M F(\nu)d\mu_g\right)^{\frac{n-1}{n}}.
\end{eqnarray*}
This is exactly \eqref{AFI} we desired. Equality holds if and only if equality in \eqref{vv} holds, whence $M$ is anisotropic umbilic, that is, $M$ is a rescaling and translation of $\w$.

For general $F$-mean convex hypersurface, inequality \eqref{AFI} follows from the approximation. The same argument in \cite{GL} shows an $F$-mean convex hypersurface which attains the equality must be strictly  $F$-mean convex hypersurface. Thus it must be a rescaling and translation of $\w$. The proof is complete.

\

\section{Discussion on general inverse anisotropic flows}\label{sec6}

By virtue of Gerhardt and Urbas' result and Guan-Li's result on the Alexandrov-Fenchel inequality, it is natural to consider 
\begin{eqnarray}\label{iacf}
\p_t X=\frac{1}{f(\k^F)}\nu_F,
\end{eqnarray}
for general positive speed function $f\in C^0(\overline{\Gamma}) \cap C^2({\Gamma})$, where $\Gamma$ is  some convex cone containing the positive cone. Assume $f$ satisfies the following conditions: 
\begin{eqnarray*}
&\hbox{(i)}& f \hbox{ is homogeneous of degree one on }\Gamma,\\
&\hbox{(ii)}& f\hbox{ is monotone, i.e. }\frac{\partial f}{\partial \l_i}>0\hbox{ on }\Gamma,\\
&\hbox{(iii)}& f\hbox{ is concave, i.e. } \frac{\partial^2 f}{\partial\l_i\partial \l_j}\leq 0 \hbox{ on }\Gamma,\\
&\hbox{(iv)}& f=0 \hbox{ on }\partial \Gamma.\\
&\hbox{(v)}& f(1,\cdots, 1)=1.
\end{eqnarray*}

We are able to show the estimate up to $C^1$ for \eqref{iacf} with $f$ satisfying (i)-(v) using our reformulation.
Denote by $f^{ij}=\frac{\p f}{\p \hat{h}_{ij}}.$

The $C^0$ estimate follows directly by the comparison principle as in Proposition \ref{C0}.

For the $C^1$ estimate, we still look at the evolution equation for $\tilde{\hat{u}}$ for $\tilde{X}=e^{-t}X$. By similar computation as in Proposition \ref{evolv}, we have
\begin{eqnarray*}
\p_t \tilde{\hat{u}}-\frac{1}{f^2}f^{ij}\left(\hat{\nabla}_i \hat{\nabla}_j \tilde{\hat{u}}+A_{pij}\hat{\nabla}^p \tilde{\hat{u}}\right)= \left(\frac{1}{f^2}f^{ij}\hat{h}_{ik}\hat{h}_{j}^k-1\right)\tilde{\hat{u}}\geq 0.
\end{eqnarray*}
The same argument as in Proposition \ref{C1} shows that the graph function has a  uniform $C^1$ bound. 

Unlike the case of the IAMCF, there is no quasilinear form for general $f$ and we have to estimate the $C^2$ directly. This is a quite delicate problem since the evolution equation  for either $h$ or $\hat{h}$ behaves messy due to the complexity of the anisotropic Gauss-Codazzi type equation \eqref{Gausseq} and \eqref{Codazzi}.
In \cite{X3}, we are  able to prove the $C^2$ estimate in some special cases when the initial hypersurface is assumed to be convex.
It is quite interesting to study such inverse type anisotropic flow, especially the case for $f=\frac{\sigma_{k+1}}{\sigma_{k}}$, in view of the Minkowski inequality \eqref{AF0} for general $i<j$.

\

\noindent{\bf Acknowledgments}.
This paper has been done while I undertook a CRC postdoc fellowship at the Department of Mathematics at McGill University.  I would like to thank Prof. Pengfei Guan  for stimulating discussion and suggestion on this subject. I would also like to thank Prof. Ben Andrews for his interest in this work. I am grateful to the  anonymous referee for his/her critical reading and valuable comments.

\


\begin{thebibliography}{999}


\bibitem{An1} Andrews, B., Evolving convex hypersurfaces, PhD Thesis, Australian National University, 1993.

\bibitem{An'} Andrews, B., Harnack inequalities for evolving hypersurfaces, Math. Z. 217 (1994), no. 1,179-197.

\bibitem{An2} Andrews, B.,  Evolving convex curves, Calc. Var. 7 (1998), 315-371.

\bibitem{An3} Andrews, B., Motion of hypersurface by Gauss curvature, Pacific J. Math. 195 (2000), pp. 1-34.

\bibitem{An4} Andrews, B., Volume-preserving anisotropic mean curvature flow. Indiana Univ. Math. J. 50 (2001), no. 2, 783-827.

\bibitem{AG1} Angenent, S., Gurtin, M. E., Multiphase thermomechanics with interfacial structure. II. Evolution of an isothermal interface. Arch. Rational Mech. Anal. 108 (1989), no. 4, 323-391.
\bibitem{AG2} Angenent, S., Gurtin, M. E., Anisotropic motion of a phase interface. Well-posedness of the initial value problem and qualitative properties of the interface. J. Reine Angew. Math. 446 (1994), 1-47.

\bibitem{BF}  Bonnessen, T., Fenchel, W., Theorie der Konvexen K\"orper, Springer Press, 1934.

\bibitem{BHW} Brendle S., Hung, P.-K., Wang M.-T., A Minkowski inequality for hypersurfaces in the Anti-deSitter-Schwarzschild manifold,   Comm. Pure Appl. Math. 69 (2016), no. 1, 124-144.

\bibitem{BW} Brendle, S.; Wang, M.-T., A Gibbons-Penrose inequality for surfaces in Schwarzschild spacetime. Comm. Math. Phys. 330 (2014), no. 1, 33-43. 

\bibitem{CW}  Chang S. -Y., Wang Y. , Inequalities for quermassintegrals on $k$-convex domains, Adv. Math., 248 (2013), 335-377.

\bibitem{CGG}  Chen, Y. G., Giga, Y., Goto, S., Uniqueness and existence of viscosity solutions of generalized mean curvature flow equations. J. Differential Geom. 33 (1991), no. 3, 749-786.

\bibitem{Ch} Chow, B.,  Deforming convex hypersurfaces by the square root of the scalar curvature. Invent. Math. 87 (1987), no. 1, 63-82.

\bibitem{CZ1} Chou K.-S. , Zhu, X.-P.,  Anisotropic flows for convex plane curves. Duke Math. J. 97 (1999), no. 3, 579-619.

\bibitem{CZ2}Chou K.-S. , Zhu, X.-P.,  The curve shortening problem. Chapman  Hall/CRC, Boca Raton, FL, 2001. x+255 pp. 

\bibitem{Di}  Ding, Q., The inverse mean curvature flow in rotationally symmetric spaces. Chin. Ann. Math. Ser. B 32 (2011), no. 1, 27-44.

\bibitem{Gage} Gage, M. E., Evolving plane curves by curvature in relative geometries,  Duke Math. J. 72 (1993) no.2, 441-466.

\bibitem{GaLi} Gage, M. E., Li, Y., Evolving plane curves by curvature in relative geometries. II. 
Duke Math. J. 75 (1994), no. 1, 79-98. 

\bibitem{Ge1} Gerhardt C.,  Flow of nonconvex hypersurfaces into spheres. J. Differential Geom. 32 (1990), no. 1, 299-314.

\bibitem{Ge2}  Gerhardt, C., Inverse curvature flows in hyperbolic space. J. Differential Geom. 89 (2011), no. 3, 487-527. 

\bibitem{Ge3}  Gerhardt, C., Curvature flows in the sphere. J. Differential Geom. 100 (2015), no. 2, 301-347. 

\bibitem{Gi} Giga, Y., Surface evolution equations. A level set approach. Monographs in Mathematics, 99. Birkh\"auser Verlag, Basel, 2006. xii+264 pp.

\bibitem{GT} Gilbarg, D., Trudinger, N. S., Elliptic partial differential equations of second order. Reprint of the 1998 edition. Classics in Mathematics. Springer-Verlag, Berlin, 2001. xiv+517 pp.

\bibitem{GL} Guan, P., Li, J., The quermassintegral inequalities for k-convex starshaped domains. Adv. Math. 221 (2009), no. 5, 1725-1732.

\bibitem{GMTZ} Guan, P., Ma, X.-N., Trudinger, N.-S., Zhu, X.-H., A form of Alexandrov-Fenchel inequality 
Pure and Applied Mathematics Quarterly,  6 (2010), 999-1012.

\bibitem{Gu} Gurtin, M. E., Thermomechanics of evolving phase boundaries in the plane. Oxford Mathematical Monographs. The Clarendon Press, Oxford University Press, New York, 1993. xi+148 pp.

\bibitem{HL} He, Y., Li, H.,  Stability of hypersurfaces with constant (r+1)-th anisotropic mean curvature. Illinois J. Math. 52 (2008), no. 4, 1301-1314. 

\bibitem{HI1} Huisken, G., Ilmanen, T., The inverse mean curvature flow and the Riemannian Penrose inequality. J. Differential Geom. 59 (2001), no. 3, 353-437.
 
\bibitem{HI2}  Huisken, G., Ilmanen, T., Higher regularity of the inverse mean curvature flow. J. Differential Geom. 80 (2008), no. 3, 433-451.
  
\bibitem{LSU}  Ladyzenskaja, O. A., Solonnikov, V. A., Ural`ceva, N. N., Linear and quasilinear equations of parabolic type. (Russian) Translated from the Russian by S. Smith. Translations of Mathematical Monographs, Vol. 23 American Mathematical Society, Providence, R.I. 1968 xi+648 pp. 





\bibitem{Re1}  Reilly, R. C., Variational properties of functions of the mean curvatures for hypersurfaces in space forms. J. Differential Geometry 8 (1973), 465-477.

\bibitem{Re2} Reilly, R. C., The relative differential geometry of nonparametric hypersurfaces. Duke Math. J. 43 (1976), no. 4, 705-721.

\bibitem{Sc} Schneider, R., Convex bodies: the Brunn-Minkowski theory, Second expanded edition. Cambrige University Press, 2014.

\bibitem{Sch} Scheuer, J., The inverse mean curvature flow in warped cylinders of non-positive radial curvature, Adv. Math. 306 (2017), 1130-1163. 

\bibitem{Shen} Shen Z., Lectures on Finsler geometry, World Scientific Publishing Co., Singapore, 2001.

\bibitem{Tr} Trudinger N. S., Isoperimetric inequalities for quermassintegrals, Ann. Inst. H. Poincar\'e Anal. Non Lin\'eaire
11 (1994), no.4, 411-425.




\bibitem{Ur} Urbas J.,  On the expansion of starshaped hypersurfaces by symmetric functions of their principal curvatures. Math. Z. 205 (1990), no. 3, 355-372. 

\bibitem{X1} Xia, C., On a class of anisotropic problem. PhD Thesis, Albert-Ludwigs University Freiburg, 2012.

\bibitem{X2} Xia, C., On an anisotropic Minkowski problem. Indiana Univ. Math. J. 62 (2013), no. 5, 1399-1430. 

\bibitem{X3} Xia, C., Inverse anisotropic curvature flow from convex hypersurfaces, J. Geom. Anal. (2016). online first. doi:10.1007/s12220-016-9755-2.

\end{thebibliography}
\end{document}